\documentclass[11pt]{amsart}

\usepackage[T1]{fontenc}
\usepackage[utf8]{inputenc}

\usepackage{enumitem}

\usepackage[matrix,arrow,curve]{xy}
\usepackage{geometry}
\usepackage{graphicx}

\usepackage[english]{babel}

\usepackage{amssymb,mathptmx
}

\usepackage[svgnames]{xcolor}

\input{macros-perso-fm-en.sty}

\newtheorem{thm}{Theorem}
\newtheorem*{thm*}{Theorem}
\newtheorem*{thmmain}{Main Theorem}

\newtheorem*{prop*}{Proposition}

\newtheorem*{lem*}{Lemma}

\newtheorem*{cormain}
{Main Corollary}

\newtheorem*{dfn*}{Definition}

\newtheorem*{rem*}{Remark}

\newtheorem{ex}
{Example}

\begin{document}

\title[Birational involutions of the real projective plane]{Birational involutions of the real projective plane fixing an irrational curve}
\author[Frédéric Mangolte]{Frédéric Mangolte\\ {\tiny joint work with I. Cheltsov, E. Yasinsky, S. Zimmermann}\\
{\tiny \today}
}

\thanks{
This review is an elaboration of a presentation given at the 
\textit{Real algebraic geometry and singularities
conference in honor of Wojciech Kucharz's 70th birthday}
 in Krakow in 2022.}

\maketitle

\section{Examples of birational involutions of the projective plane $\PP^2$}
Let $\Bbbk=\RR,\CC$.

The first way to describe an element of $\Bir_\Bbbk(\PP^2)$ is to explicitly express it  in homogeneous coordinates.
\begin{ex}[Standard Cremona involution]
The birational map from $\PP^2$ to $\PP^2$ given by
$$
\alpha_0\colon [x:y:z]\dashrightarrow [yz: xz: xy]
$$
is called the \emph{Standard Cremona involution}.
It is well-defined except at the three points $[1:0:0],[0:1:0],[0:0:1]$ \emph{(base points)} and is conjugate in $\Bir_\Bbbk(\PP^2)$ to the linear involution 
$$
\tau_0\colon  [x:y:z]\dashrightarrow [x: y: -z]\;.
$$

\end{ex}

Another way to define a birational involution of $\PP^2$ is to start with a biregular involution on a smooth rational surface which is called a \emph{model}, and then pull it back to $\PP^2$ by a birational map. 
\begin{ex}[Geiser involution]
\label{ex.dp2}
Let $S_2$ be a complex del Pezzo surface of degree $2$ with anti-canonical double cover 
$$
 \xymatrix{\pi\colon
S_2\ar@{->}[r]_{}^{2:1}&\PP^2
}
$$
 with branch locus a smooth quartic plane curve $C$. The surface $S_2$ is defined by an equation of the form
 $$
 w^2=f_4(x,y,,z)
 $$
 in $\PP(2,1,1,1)$ where $f_4$ is the homogeneous polynomial of degree $4$ whose zero locus in $\PP^2$ is the curve $C$.

The biregular involution $\tau\in\Aut_\CC(S_2)$ given by
$$
\tau\colon [w:x:y:z]\dashrightarrow [-w:x:y:z]
$$
exchanges the two sheets of the double cover. The surface $S_2$ is rational over $\CC$ which means there is a birational map
$$
\varphi\colon S_2 \dasharrow \PP^2\;.
$$

Hence the composed map $\alpha= \varphi\tau\varphi^{-1}$ is an order $2$ element of the group $\Bir_\CC(\PP^2)$ classically called the \emph{Geiser} involution.

Assume now that $S_2$ is defined over $\RR$ (\textit{i.e.} $f_4$ has real coefficients). If the real locus $S_2(\RR)$ is non empty and connected (for the euclidean topology), then $S_2$ is rational over $\RR$ by Comessatti's Theorem (see e.g. \cite{ma-book-en}),  and we can assume that $\varphi$ is defined over $\RR$.

The biregular automorphism $\tau\in\Aut_\RR(S_2)$ leads to a birational involution $\alpha\in\Bir_\RR(\PP^2)$.
\end{ex}

Observe that $\alpha$ is only defined up to conjugation in $\Bir_\Bbbk(\PP^2)$ as the map $\varphi$ is not uniquely defined. when we say "the" Geiser involution, we speak in fact of the conjugacy class of $\alpha$ in $\Bir_\Bbbk(\PP^2)$.

\section{Reduction to $G$-birational classification of minimal del Pezzo surfaces and conic bundles}
\label{section.mmp}
In fact we can reverse and generalize the process used in the second example by \emph{regularizing} any birational involution of $\PP^2$.

From now on, we let $G=\ZZ/2$.

\begin{prop*}
 Let $\Bbbk=\RR,\CC$ and $\alpha\in\Bir_\Bbbk(\PP^2)$ be an element of order $2$. There exists a smooth rational surface $S$ and a birational map defined over $\Bbbk$
 $$
\varphi\colon S \dasharrow \PP^2\;.
$$
 such that 
 $\tau:=\varphi^{-1}\alpha\varphi\in\Aut_\Bbbk(S)$ is a biregular involution of $S$.
\end{prop*}

\begin{dfn*}
Let $S$ be a smooth
surface  
endowed with a nontrivial biregular involution $\tau$. 
The subgroup $H=\langle\tau\rangle\subset\Aut_\Bbbk(S)$ is then isomorphic to $G$ and the pair $(S,H)$ is called a \emph{$G$-surface}.
\end{dfn*}

\begin{prop*}
 Let $\alpha,\alpha'\in\Bir_\Bbbk(\PP^2)$ be elements of order $2$, $(S,\langle\tau\rangle)$ and $(S',\langle\tau'\rangle)$ associated rational $G$-surfaces. Then $\alpha,\alpha'$ are conjugated in $\Bir_\Bbbk(\PP^2)$ if and only if the associated $G$-surfaces  are equivariantly birational. That is there exists a birational map $\varphi \colon S \dasharrow S'$ such that $\varphi\tau\varphi^{-1}=\tau'$ over a Zariski dense open subset of $S$.
\end{prop*}

Hence to classify conjugacy classes of elements of order $2$ in $\Bir_\Bbbk(\PP^2)$, we classify equivariant birational classes of rational $G$-surfaces. For this purpose, take a $G$-surface $(S,\langle\tau\rangle)$ rational over $\Bbbk$ and run a $G$-MMP over $\Bbbk$ (see \cite{Ko-surf}) which ends with a pair $(S^*,\langle\tau^*\rangle)$. There are two possibilities for $S^*$:
\begin{prop*}
Let $\Bbbk=\RR,\CC$ and $(S,H)$ be a $G$-surface rational over $\Bbbk$. Denote by $(S^*,H^*)$ the output of a $G$-MMP over $\Bbbk$. Then $S^*$ belong to one of the two following classes:
\begin{enumerate}
\item[(DP)] $S^*$ is a del Pezzo surface such that $\Pic^G(S^*)\simeq \ZZ$\;;
\item[(CB)] $S^*$ admits a $G$-conic bundle structure over $\PP^1$ and $\Pic^G(S^*)\simeq \ZZ^2$\;;
\end{enumerate}

Here the action of $G$ on $\Pic(S^*)$ is given by $H^*$.
\end{prop*}

When the hypothese on the invariant part of the Picard group is satisfied we say the $G$-surface, DP or CB,  is \emph{minimal}.

The initial problem of classification of conjugacy classes of birational involutions is now reduced to the $G$-equivariant birational classification of minimal $G$-surfaces belonging to the set $(DP)\cup (CB)$. 
In fact in \cite{cmyz} we went further and gave explicit models of all such pairs.

The two former examples are in $(DP)$: $(S,H)=(\PP^2,\alpha_0)$ and $(S^*,H^*)=(\PP^2,\tau_0)$ for the first example; $(S,H)=(\PP^2,\alpha)$ and $(S^*,H^*)=(S_2,\tau)$ in the second example.

\section{Main invariant: the fixed curve}

Recall that a real variety $X$ is \emph{geometrically rational} if its complexification $X_\CC$ is rational. For example, a smooth geometrically rational real curve $C$ is rational if and only if $C(\RR)\ne\varnothing$. A complex variety is rational if and only if it is geometrically rational.
\begin{prop*}
Let $(S,H)$ and $(S',H')$ be $G$-surfaces and $\varphi \colon S \dasharrow S'$ a $G$-equivariant rational map. 
  If $C$ is a geometrically irrational curve on $S$, its proper transform $C':=\varphi(C)$ is a geometrically irrational curve on $S'$. If furthermore $C$ is fixed by $H$ then $C'$ is fixed by $H'$. If $\varphi$ is birational, the curves $C$ and $C'$ are birational. They are isomorphic if they are smooth. 
\end{prop*}
\begin{proof}
The proper transform $C'$ is obtained from $C$ in the following way: let $C_0$ be the image of $\varphi$ of the open subset of $C$ where $\varphi$ is defined. The set $C_0$ is a curve because $\varphi$ contracts only geometrically rational curves. Then let $C'$ be the Zariski closure of $C_0$ in $S'$.
\end{proof}

\begin{dfn*}
Let $S$ be a rational surface over $\Bbbk$ and $\tau\in\Aut_\Bbbk(S)$ an element of ordre $2$. Define $F(\tau)$ the normalization of the union of geometrically irrational curves fixed by $\tau$. In particular, $F(\tau)=\varnothing$ if $\tau$ fixes no geometrically  irrational curve.
\end{dfn*}

From the discussion above, $F(\tau)$ is a conjugacy invariant.

\begin{rem*}
 In fact we can prove that in our context, we get exactly two cases, see \cite[Lemma 2.7]{cmyz}:
 \begin{enumerate}
\item $F(\tau)=\varnothing$, or
\item $F(\tau)=C$ where $C$ is a smooth geometrically irreducible curve of genus $g\geq 1$.
\end{enumerate}
\end{rem*}

Returning to the two former examples, we get $F(\tau)=\varnothing$ for the standard Cremona involution and $F(\tau)$ is the smooth non hyperelliptic curve of genus $3$ given by the equation $f_4=0$ in $\PP^2$ in the second example.

\section{Main result}

\subsection{Classification over $\CC$}

Using an equivariant version of Mori theory in dimension two as discussed above, L. Bayle and A. Beauville obtained a very precise classification (see \cite{BB00}):
\begin{thm*}[Bayle-Beauville 2000]
Let $\alpha$ be an element of order $2$ in the group $\Bir_\CC(\PP^2)$. Then $\alpha$ is conjugate in $\Bir_\CC(\PP^2)$ to one and only one of the following $4$ classes of involution:
\begin{enumerate}
\item The linear involution on $\PP^2$ given by
$
\tau_0\colon  [x:y:z]\dashrightarrow [x: y: -z]\;.
$
 
 $F(\alpha)=\varnothing$.
 \item A Bertini involution (analogously to Example~\ref{ex.dp2}, a biregular model is the deck involution of a del Pezzo surface of degree~$1$ given by a double cover of the quadric cone with branch locus the fixed curve of the involution).
  
  $F(\alpha)$ is a smooth non hyperelliptic curve of genus $g=4$ canonically  embedded in a quadric cone.
\item A Geiser involution (see Example~\ref{ex.dp2}). 
 
 $F(\alpha)$ is a smooth non hyperelliptic curve of genus $g=3$.
\item A de Jonqui\`{e}res involution (see Section~\ref{section.jonquieres}). 

$F(\alpha)$ is a smooth hyperelliptic curve of genus $g\geq 1$.
\end{enumerate}
\end{thm*}

Except for the case (1), all these involutions have moduli \cite{BB00}.
Namely, conjugacy classes of de Jonqui\`{e}res involutions of genus $g\geqslant 1$ are parametrized by hyperelliptic curves of genus~$g\geqslant 1$.
Conjugacy classes of Geiser involutions are parametrized by non-hyperelliptic curves of genus~$3$,
and conjugacy classes of Bertini involutions are parametrized by non-hyperelliptic curves of genus~$4$ canonically  embedded in a quadric cone.

\subsection{Classification over $\RR$}
The first new involution in this context is the antipodal map on the quadric sphere.
In \cite{cmyz}, we discovered $7$ additional classes of involutions in $\Bir_\RR(\PP^2)$ and called them $d$-twisted Trepalin involutions, $d=0,1,2$, Kowalevskaya involution and $d$-twisted Iskovskikh involutions $d=0,1,2$, see Section~\ref{section.isk} for Iskovskikh involutions and \cite{cmyz} for the definitions of the others.

\begin{thmmain}[Cheltsov-Mangolte-Yasinsky-Zimmermann 2024]
\label{theorem:CMYZ}
Let $\alpha$ be an element of order $2$ in the group $\Bir_\RR(\PP^2)$. Then $\alpha$ is conjugate in $\Bir_\RR(\PP^2)$ to one of the following $12$ classes of involution:
\begin{enumerate}
\item The linear involution $\tau_0$ or the antipodal involution on the quadric sphere or a $t$-twisted Trepalin involution, $t=0,1,2$. 

$F(\alpha)=\varnothing$: 
 \item A Bertini involution.
 
 $F(\alpha)$ is a smooth non hyperelliptic curve of genus $g=4$ canonically  embedded in a quadric cone. 
 \item A Geiser involution. 
 
 $F(\alpha)$ is a smooth non hyperelliptic curve of genus $g=3$. 
  \item A Kowalevskaya involution. 
 
 $F(\alpha)$ is a smooth elliptic curve. 
 \item A de Jonqui\`{e}res involution or a $t$-twisted Iskovskikh involutions $t=0,1,2$ (see Section~\ref{section.isk}).
 
$F(\alpha)$ is a smooth hyperelliptic curve of genus $g\geq 1$.
 \end{enumerate}
\end{thmmain}

Furthermore, involutions in different classes are not conjugate except some exceptions when the fixed curve is elliptic, see \cite[Main Theorem]{cmyz} for details.

In contrast with the complex case, fixed curves does not parametrize conjugacy classes. See \cite[Main Corollary]{cmyz}

\begin{cormain}
 Let $C$ be a~real smooth projective hyperelliptic curve of genus $g\geqslant 2$ such that the~locus $C(\RR)$ consists of at least $2$ connected components.
Then $\Bir_\RR(\PP^2)$ contains uncountably many non-conjugate involutions that all fix a~curve isomorphic to the~curve $C$. Besides, the real plane Cremona group $\Bir_\RR(\PP^2)$ contains uncountably many non-conjugate involutions that fix no geometrically irrational curves.
\end{cormain}

\section{Birational models of $G$-conic bundles over $\RR$}
\label{section.cb}
To illustrate the methods used in the proof of the main theorem above, we will focus now on the case where  $S^*$ is in the case $(CB)$, see Section~\ref{section.mmp}.

Let $\Bbbk=\RR,\CC$ and $S$ be a smooth surface defined over $\Bbbk$ endowed with a biregular involution $\tau$. Assume that $S$ admits a $G$-equivariant morphism $\pi\colon S\to\PP^1$ whose fibers are conics. Assume furthermore that $\Pic^G(S^*)\simeq \ZZ^2$.  We have the following, \cite[Lemma 2.7]{cmyz} for a proof.

\begin{lem*}
If $C=F(\tau)$ is a geometrically irrational curve, then $\tau$ acts trivially on the base $\PP^1$ and $C$ is a double section of $\pi$.
\end{lem*}

Assume from now on that $\tau$ acts trivially on the base of the $G$-conic bundle $\pi\colon S\to\PP^1$. In this case, a general complex fiber of $\pi$ is a smooth conic on which $\tau$ restricts to an involution and there is a finite number of singular fibers which are unions of two smooth complex rational curves $F_1$, $F_2$ intersecting transversally in one point. Each $F_i$, $i=1,2$ is a $(-1)$-curve on the complexification $S_\CC$ of $S$.

Over $\CC$, for each singular fiber, we must have $\tau(F_i)=F_{3-i}$ because $\Pic^G(S^*)\simeq \ZZ^2$.

Over $\RR$, denoting by $\sigma$ the real structure on the complexification $S_\CC$ ($\sigma$ is an anti-holomorphic involution on $S_\CC$) , at least one of the two involutions $\tau$ or $\sigma$ must exchanges $F_1$ and $F_2$ for the same reason.

\subsection{De Jonqui\`{e}res involutions}
\label{section.jonquieres}
Firstly assume that $\pi$ admits a section $Z$ defined over $\Bbbk$. Then $Z+\tau(Z)$ is $G$-invariant and  there exists a $G$-equivariant birational map $\chi\colon S\dasharrow X$ that fits into the following commutative $G$-equivariant diagram:
$$
\xymatrix{
S\ar@{->}[d]_{\pi}\ar@{-->}[rr]^\chi&& X\ar@{->}[d]^{\eta}\ar@{->}[rr]^{\rho}&&Y\ar@{-->}[dll]\\
\PP^1\ar@{=}[rr]&&\PP^1&&}
$$
where $X$ is a smooth surface, $\eta$ is a conic bundle such that $\Pic(X)^G\simeq\mathbb{Z}^2$,
$Y$ is a hypersurface in $\mathbb{P}(d,d,1,1)$ of degree $2d=8-K_S^2$ that is given  by
\begin{equation*}
xy=f(z,t)
\end{equation*}

for some homogeneous polynomial $f(z,t)$ of degree $2d$ that has no multiple roots. The map $Y\dasharrow\PP^1$ is given by
$$
\big[x:y:z:t\big]\mapsto \big[z:t\big],
$$
where $x$, $y$, $z$ and $t$ are coordinates on $\PP(d,d,1,1)$ of weights $d$, $d$, $1$ and $1$, respectively.

The curves $Z$ and $\tau(Z)$ are $\rho\circ\chi$-exceptional, the involution $\tau$ acts on the surface $Y$ as
$$
\big[x:y:z:t\big]\mapsto \big[y:x:z:t\big],
$$
and the morphism $\rho$ is a minimal resolution of singularities.

The fixed locus of $\tau$ is the 
curve $C\simeq\rho(C)$, where $\rho(C)$ is given by
$$
\left\{\aligned
&x=y,\\
&x^2=f(z,t).
\endaligned
\right.
$$

If $d\geqslant 3$, then $\rho(C)$ is a real hyperelliptic curve of genus $g=d-1$ with hyperelliptic covering
\[
\nu\colon C\to\PP^1,\quad [x:y:z:t]\mapsto [z:t].
\]

Similarly, if $d=2$, then $\rho(C)$ is an elliptic curve. The number of real roots $r$ of $f$ is even and the number of connected components of $C(\RR)$ is $\frac12 r$.

If $\Bbbk=\CC$, forgetting the action of $G$, we can always contract one of the $(-1)$-curves in any singular fiber of $\pi$ and obtain a locally trivial $\PP^1$-fibration $S'\to\PP^1$ ($S'$ is an Hirzebruch surface). Any such fibration has a complex section whose pullback $Z$ is a section of $\pi\colon S\to\PP^1$. Hence any $G$-surface for which $S^*$ is in the case $(CB)$ admits such a model $Y\dasharrow\PP^1$. In this case, $\tau$ is called a \emph{de Jonqui\`{e}res} involution.

If $\Bbbk=\RR$, we cannot contract a $(-1)$-curve if it's not defined over $\RR$ (case $\sigma(F_i)=F_{3-i}$). 
So we need to consider another model when $\pi\colon S\to\PP^1$ has no real section.

As a step in the classification, we get the following characterization of de Jonqui\`{e}res involutions over $\RR$, see \cite[Proposition~6.5]{cmyz}.

\begin{prop*}
Let $S$ be a real rational surface and $\pi\colon S\to \PP^1$ a $G$-conic bundle with biregular involution $\tau$ acting trivially on the base and such that $\Pic^G(S)\simeq \ZZ^2$. Then $\tau$ is a de Jonqui\`{e}res involution if and only if 
$$
\pi(S(\RR))=\PP^1(\RR)\approx S^1\;.
$$
\end{prop*}

\subsection{$d$-twisted Iskovskikh involutions}
\label{section.isk}

In the case $\pi(S(\RR))\subsetneq\PP^1(\RR)$, we prove first the existence of a good model in Theorem~\ref{thm.model.existence}, then its unicity in Theorem~\ref{thm.model.unicity}, see \cite[Theorems 7.1 and 7.6]{cmyz}.

\begin{thm}
\label{thm.model.existence}
 Let $\pi\colon S\to \PP^1$ be a minimal real rational $G$-conic bundle with biregular involution $\tau$ acting trivially on the base. 
 Assume that $\pi(S(\RR))\subsetneq\PP^1(\RR)$.
Then there exists $G$-equivariant commutative diagram
\begin{equation*}
\xymatrix{
S\ar@{->}[d]_{\pi}\ar@{-->}[rr]^{\chi}&&X\ar@{->}[d]^{\eta}\\
\PP^1\ar@{->}[rr]_\phi&&\PP^1}
\end{equation*}
where $\chi$ is a birational map, $\phi\in\mathrm{PGL}_2(\mathbb{R})$, $X$ is a smooth surface, $\eta$ is a $G$-minimal conic bundle,
the fiber $\eta^{-1}([1:0])$ is smooth and does not have real points,
the quasi-projective surface $Y=X\setminus \eta^{-1}([1:0])$ is given in $\mathbb{P}^2\times\mathbb{A}^1$ with coordinates $([x:y:z],t)$ 
\begin{equation*}
A(t)x^2+B(t)xy+C(t)y^2=H(t)z^2
\end{equation*}
for some polynomials $A,B,C,H\in\mathbb{R}[t]$ such that $\Delta=(B^2-4AC)H$ does not have multiple roots and $\deg(\Delta)$ is even,
the involution $\tau$ acts on the surface $Y$ by
$$
([x:y:z],t)\mapsto([x:y:-z],t),
$$
and the restriction map $\eta\vert_{Y}\colon Y\to\mathbb{P}^1\setminus [1:0]=\mathbb{A}^1$ is the map given by $([x:y:z],t)\mapsto t$.
Moreover, the following holds:
\begin{itemize}
\item the polynomial $H(t)$ has only real roots and its leading coefficient is negative,
\item fibers of $\eta$ over roots of the polynomial $H(t)$ are singular irreducible conics ($\sigma(F_i)=F_{3-i}$).
\end{itemize}
\end{thm}

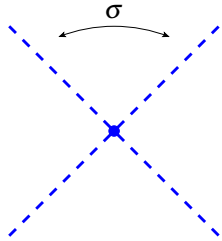
\begin{figure}[!htb]
\begin{picture}(0,0)%
\includegraphics{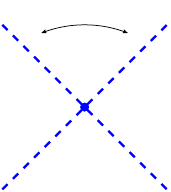}%
\end{picture}%
\setlength{\unitlength}{1579sp}%
\begingroup\makeatletter\ifx\SetFigFont\undefined%
\gdef\SetFigFont#1#2#3#4#5{%
  \reset@font\fontsize{#1}{#2pt}%
  \fontfamily{#3}\fontseries{#4}\fontshape{#5}%
  \selectfont}%
\fi\endgroup%
\begin{picture}(3381,3799)(5098,-7639)
\put(6789,-4179){\makebox(0,0)[b]{\smash{{\SetFigFont{10}{12.0}{\rmdefault}{\mddefault}{\updefault}{\color[rgb]{0,0,0}$\sigma$}%
}}}}
\end{picture}%
\caption{Singular irreducible real fiber}  
\end{figure}

\begin{dfn*}
In the assumptions of Theorem~\ref{thm.model.existence}, let $d=\deg H$. We call $\tau$ a
 \emph{$d$-twisted Iskovskikh involution}.
\end{dfn*}

The fixed curve $C=F(\tau)$ is  given by 
\begin{equation}
\label{equation.hyper}
 w^2=B^2-4AC\;.
\end{equation}

It is elliptic if $\deg(B^2-4AC)=4$ and hyperelliptic if $\deg(B^2-4AC)\geq6$. 

Indeed, the fixed curve is given by $z=0$ which gives $Ax^2+Bxy+Cy^2=0$. Letting $w:=2(\frac xy A+\frac12 B)$ we get \eqref{equation.hyper}.

\begin{thm}
\label{thm.model.unicity}
In the assumptions of Theorem~\ref{thm.model.existence}, two $G$-conic bundles $\eta\colon X\to \PP^1,\eta\colon X'\to \PP^1$ are $G$-equivariantly birational if and only if
\begin{enumerate}
\item They have same discriminant loci $\{\Delta=0\}=\{\Delta'=0\}$\;;
\item They have the same real interval $\eta(X(\RR))=\eta'(X'(\RR))$.
\item Sign conditions: $B^2-4AC=\lambda ({B'}^2-4A'C')$ and $H=\mu H'$ for some positive real numbers $\lambda, \mu$.
\end{enumerate}

Note that the third condition implies the first one. I put them like this for didactical purposes.
\end{thm}

We conclude this note by giving explicit proof of the main corollary in the case $C$ irrational (see \cite[Section 8.C]{cmyz} for details).

Let $C$ be an hyperelliptic curve given by
$$
w^2=-4f(t)
$$
where  $f\in\RR[t]$ has even degree $\geq 6$, only simple roots and at least $4$ real roots. For given real numbers $a,b$, let $S_{a,b}$ be the surface with equation
$$
x^2+f(t)y^2=-(t-a)(t-b)z^2\;.
$$

Let $\tau_{a,b}$ be the corresponding involution. Then for general $a,b,a',b'$ in a given interval (we need to preserve the connectedness of the real locus of the surface),  $\tau_{a,b}$ and $\tau_{a',b'}$ are not conjugate by Theorem~\ref{thm.model.unicity}.

\nocite{ma-book,ma-book-en,cmyz}
\bibliographystyle{amsalpha}
\def\cprime{$'$} \def\cprime{$'$} \def\cprime{$'$} \def\cprime{$'$}
  \def\cprime{$'$} \def\cprime{$'$} \def\cprime{$'$}
\providecommand{\bysame}{\leavevmode\hbox to3em{\hrulefill}\thinspace}
\providecommand{\MR}{\relax\ifhmode\unskip\space\fi MR }
\providecommand{\MRhref}[2]{%
  \href{http://www.ams.org/mathscinet-getitem?mr=#1}{#2}
}
\providecommand{\href}[2]{#2}

 \end{document}